\documentclass[10pt]{amsart}

\usepackage[headings]{fullpage}

\usepackage{amsfonts,amssymb,amsthm,amsmath,euscript}

\usepackage[colorlinks,citecolor=blue,linkcolor=red]{hyperref}

\usepackage{graphicx,setspace,amscd,float,hyperref,enumitem} 

\newtheorem{thm}{Theorem}[section]
\newtheorem{lem}[thm]{Lemma}

\newtheorem{prop}[thm]{Proposition}
\newtheorem{cor}[thm]{Corollary}
\theoremstyle{definition}
\newtheorem{df}[thm]{Definition}
\newtheorem{rk}[thm]{Remark}

\newtheorem{conv}[thm]{Convention}

\newtheorem{notation}[thm]{Notation}

\newtheorem{theor}{Theorem}

\def\epsilon{\varepsilon}
\def\phi{\varphi}

\newcommand{\Out}{\mbox{Out}}

\begin{document}

\title[On purely loxodromic actions]{On purely
  loxodromic actions}

\author{Ilya Kapovich}
\address{\tt  Department of Mathematics, University of Illinois at Urbana-Champaign,
  1409 West Green Street, Urbana, IL 61801
  \newline \indent  http://www.math.uiuc.edu/\~{}kapovich, } \email{\tt kapovich@math.uiuc.edu}

\begin{abstract}
We construct an example of an isometric action of $F(a,b)$ on a
$\delta$-hyperbolic graph $Y$, such that this action is acylindrical,
purely loxodromic, has asymptotic translation lengths of nontrivial
elements of $F(a,b)$ separated away from $0$, has quasiconvex orbits in $Y$, but such that the orbit
map $F(a,b)\to Y$ is not a quasi-isometric embedding.
\end{abstract}

\thanks{The author was partially supported by the NSF grant DMS-1405146.}

\subjclass[2010]{Primary 20F65, Secondary  57M}

\date{}
\maketitle

\section{Introduction}

There are many natural situations in geometric topology and geometric
group theory when one wants to understand, given a group $G$
acting on some Gromov-hyperbolic space $X$, and a finitely generated
``purely loxodromic'' subgroup $H\le G$,  whether the orbit map $H\to X$ is a quasi-isometric embedding.  Here "purely loxodromic"  means that every element $h\in H$ of infinite order acts loxodromically on $X$.  The model example of this problem comes from studying subgroups of mapping class groups.  Let $S$ be a closed oriented hyperbolic surface and let $\mathcal C(S)$ be the curve complex of $S$ (known to be Gromov-hyperbolic by a result of Masur and Minsky~\cite{MM}). It is known that an element $g$ of the mapping class group $Mod(S)$ acts loxodromically on $\mathcal C(S)$ if and only if $g$ is pseudo-Anosov. A finitely generated subgroup $H\le Mod(S)$ is called \emph{convex cocompact} (see \cite{FM,Ha05,KeLe07,KeLe08}) if the orbit map $H\to \mathcal C(S)$ is a quasi-isometric embedding.  An important open problem in the study of mapping class groups asks whether every "purely pseudo-Anosov" (that is purely loxodromic for the action on $\mathcal C(S)$) finitely  generated subgroup of $Mod(S)$ is convex cocompact. 

Note that if $G$ is a word-hyperbolic group acting by translations on its Cayley graph $X$, then $g\in G$ is loxodromic if and only if $g$ has infinite order.  In this case whenever $H\le G$ is a finitely generated subgroup which is not quasiconvex in $G$, then $H$ is purely loxodromic but the orbit map $H\to X$ is not a quasi-isometric embedding. However, in this case the orbit of $H$ in $X$ is not a quasiconvex subset of $X$. Moreover, for a finitely generated subgroup $H\le G$ the orbit map $H\to X$ is a quasi-isometric embedding if and only if every (equivalenty, some) orbit of $H$ in $X$ is quasiconvex.  There are many examples of finitely generated (even word-hyperbolic) subgroups of word-hyperbolic groups that are not quasiconvex. For instance, if $G$ is the fundamental group of a closed hyperbolic 3-manifold $M$ fibering over the circle with fiber $S$, then $G=\pi_1(M)$ is word-hyperbolic and $\pi_1(S)\le G$ is not quasiconvex.

There are some situations where purely loxodromic subgroups do have
quasi-isometric embedding orbit maps. Thus a recent paper \cite{KMT}
of Koberda, Mangahas, and Taylor provides a result of this kind. Given
a right-angled Artin group $G=A(\Gamma)$ defined by a finite graph
$\Gamma$, there is an associated Gromov-hyperbolic graph $\Gamma^e$
(see \cite{KK13}), called the "extension graph",  which comes equipped
with a natural isometric action of $G$. They prove in \cite{KMT}  that
for a finitely generated subgroup $H\le G$ the orbit map $H\to
\Gamma^e$ is a quasi-isometric embedding if and only if the action of
$H$ on $\Gamma^e$ is purely loxodromic. This result is proved in
\cite{KMT} in the context of exploring a strong form of
quasiconvexity for finitely generated subgroups of finitely generated
groups called ``stability''. 

The group $Out(F_N)$ (where $F_N$ is a free group of finite rank $N\ge
3$) has a natural isometric action on the "free factor graph"
$\mathcal F_N$, which is known to be
Gromov-hyperbolic~\cite{BF14,KR,HiHo} and provides one of several
$Out(F_N)$  analogs of the curve complex. It is known~\cite{BF14} that
$\phi\in Out(F_N)$ acts on $\mathcal F_N$ loxodromically if and only
if $\phi$ is fully irreducible. There are two types of fully
irreducible elements of $Out(F_N)$: atoroidal ones (which have no
nontrivial periodic conjugacy classes in $F_N$ and have
word-hyperbolic mapping torus groups) and non-atoroidal ones.  It is
known~\cite{BH92} that a non-atoroidal $\phi\in Out(F_N)$ is fully
irreducible if and only if $\phi$ is induced by a pseudo-Anosov
homeomorphism of a compact surface with one boundary component. In
\cite{DT} Dowdall and Taylor proved that if a finitely generated $H\le
Out(F_N)$ is "purely atoroidal" and has the orbit map $H\to \mathcal
F_N$ being quasi-isometric embedding (which implies that $H$ is also
purely loxodromic for the action on $\mathcal F_N$) then the natural
extension $G_H$ of $F_N$ by $H$ is word-hyperbolic. Hamenstadt  and
Hensel~\cite{HH} suggested to call a finitely generated subgroup $H\le
Out(F_N)$ "convex cocompact" if the orbit map $H\to \mathcal F_N$ is a
quasi-isometric embedding. However, with this definition, an infinite
cyclic $H=\langle \phi\rangle\le Out(F_N)$, generated by a
non-atoroidal fully irreducible $\phi$, is considered convex
cocompact, although the group $G_H$ is not word-hyperbolic in this
case. Mann and Reynolds~\cite{Mann} defined a further coarsely
Lipschitz coarsely equivariant quotient $\mathcal P_N$ of $\mathcal
F_N$ such that $\mathcal P_N$ is Gromov-hyperbolic and such that
$\phi\in Out(F_N)$ acts loxodromically on $\mathcal P_N$ if and only
if $\phi$ is an atoroidal fully irreducible. In a new paper~\cite{DT1}
Dowdall and Taylor show that if $H\le Out(F_N)$ is a finitely
generated purely atoroidal subgroup such that the orbit map
$H\to\mathcal F_N$ is a quasi-isometric embedding (so that $H$ is
purely loxodromic for the action on $\mathcal P_N$) then the orbit map
$H\to\mathcal P_N$ is also a quasi-isometric embedding. This result
provides another interesting example where a purely loxodromic action
can be shown to have the orbit map being a quasi-isometric embedding
(under the initial assumption that the orbit map $H\to \mathcal F_N$ is a
quasi-isometric embedding.)

The goal of this note is to show that even if we make rather strong additional geometric assumptions about a purely loxodromic isometric action of a word-hyperbolic group $H$ on a Gromov-hyperbolic space $X$ (including discreteness and quasi convexity of  $H$-orbits), that is not enough to ensure that the orbit map $H\to X$ is a quasi-isometric embedding.

Before stating the main result, we recall several definitions. 

\begin{df}[Asymptotic translation length]
Let $G$ be a group acting isometrically on a metric space $X$. For an element $g\in G$ the \emph{asymptotic translation length} $||g||_X$ of $g$ on $X$ is
\[
||g||_X:=\lim_{n\to\infty} \frac{d_X(x,g^nx)}{n},
\]
where $x\in X$ is a basepoint.
\end{df}
It is well-known that the above limit always exists and does not depend on the choice of $x\in X$. Moreover, for an element $g\in G$, the map $\mathbb Z\to X$, $n\mapsto g^nx$, is a quasi-isometric embedding if and only if $||g||_X>0$. In particular, if $X$ is Gromov-hyperbolic, then $g\in G$ acts logodromically on $X$ if and only if $||g||_X>0$.  

\begin{df}[Acylidrical actions] 
An isometric action of a group $G$ on a Gromov-hyperbolic space $X$ is said to be \emph{acylindrical} if for every $R\ge 0$ there exist $L\ge 1$ and $M\ge 1$ such that whenever $x,y\in X$ are such that $d_X(x,y)\ge L$ then
\[
\#\left(\{g\in G| d_X(x,gx)\le R, d_X(y,gy)\le R\}\right)\le M
\]
\end{df}

Acylidrical actions on hyperbolic spaces play a crucial role in studying various generalizations of relatively hyperbolic groups, particularly the so-called acylindrically hyperbolic groups (see, for example \cite{DGO,Os13,H,GS,PS}), and in the study of group actions on $\mathbb R$-trees (see, for example, \cite{De,KW,Se,A}). The action of $Mod(S)$ on the curve complex $\mathcal C(S)$ is also known to be acylindrical, see \cite{Bow08} and this fact has many useful consequences in the study of mapping class groups. Acylindricity is a rather strong assumption, which brings some degree of finiteness to non-proper actions and also imposes substantial algebraic restrictions on the situation.  

Out main result is (c.f. Theorem~\ref{thm:main} below):

\begin{theor}\label{thm:A}

There exists a Gromov-hyperbolic graph $Y$ with a simplicial isometric action of $F(a,b)$ on $Y$ such that the following hold: 

\begin{enumerate}
\item The action of $F(a,b)$ on $Y$ is acylindrical.
\item The action of $F(a,b)$ on $Y$ is purely loxodromic, that is, every $1\ne g\in F(a,b)$ acts on $Y$ as a loxodromic isometry.
\item For every $1\ne g\in F(a,b)$ we have $||g||_Y\ge 1/7$.
\item For any $p\in Y$ the orbit $F(a,b)p\subseteq Y$ is a quasiconvex subset of $Y$.
\item There exists $C\ge 1$ such that for any $x,y\in F(a,b)$ if $\alpha_{x,y}$ is a geodesic from $x$ to $y$ in the Cayley graph of $F(a,b)$ with respect to the basis $\{a,b\}$, and if $\beta=[x,y]_Y$ is a geodesic from $x$ to $y$ in $Y$, then $\alpha$ and $\beta$ are $C$-Hausdorff close in $Y$.
\item For any $p\in Y$, the orbit map $F(a,b)\to Y$, $g\mapsto gp$, is not a quasi-isometric embedding, and, moreover, the action of $F(a,b)$ on $Y$ is not metrically proper. 
\end{enumerate}
\end{theor}

Note that, by the standard Milnor-Svarc argument (c.f. \cite[Proposition 8.19]{BH}), if $G$ is a group acting by isometries on a Gromov-hyperbolic metric space $X$ with quasiconvex orbits and if the action is metrically proper (that is, if for every metric ball $B$ the set $\{g\in G| B\cap gB\ne \varnothing\}$ is finite), then $G$ is finitely generated and the orbit map $G\to X$ is a quasi-isometric embedding.  

An instructive example for comparison with Theorem~\ref{thm:A}  comes from group actions on $\mathbb R$-trees that live in the boundary of the Culler-Vogtmann Outer space. Let $\phi\in \Out(F(a,b,c))$ be an atoroidal fully irreducible automorphism and let $T=T_\phi$ be the "stable" $\mathbb R$-tree for $\phi$, constructed from a train-track representative of $\phi$ (see~\cite{BF94,BFH97} for the construction of $T_\phi$). Then $F_3=F(a,b,c)$ acts on $T$ freely, isometrically and with dense orbits in $T$ (see, for example,  \cite{GJLL,LL}), so that this action is purely loxodromic and all $F_3$-orbits are quasiconvex in $T$. 
Condition (5) of Theorem~\ref{thm:A} also holds in this case because of the so-called "bounded back-tracking property" for "very small" actions of free groups on $\mathbb R$-trees~\cite{GJLL}.  Since the action on $T$ has dense orbits, the set of asymptotic translation lengths of nontrivial elements of $F_3$ is not separated away from $0$. The action is also not acylindrical. Indeed, take $R=1$. Then for any $M\ge 1$ there exists an element $g\in F_3$ with $0<||g||<1/M$. Consider the axis $L(g)\subseteq T$, so that $g$ acts on $L(g)$ by a translation of magnitude $||g||_T$. For any $L\ge 1$ take points $x,y\in L(g)$ with $d_T(x,y)\ge L$. Then for $k=0,1,2\dots, M$ the element $g^k$ translates each of $x$, $y$ by $k||g||_T\le 1$ so that we have $\ge M+1$ distinct elements displacing each of $x,y$ by $\le 1$. Thus the action of $F_3$ on $T$ is indeed not acylindrical. Finally, the orbit map $F_3\to T$ is not a quasi-isometric embedding. Thus this example satisfies properties (2), (4), (5) and (6) from Theorem~\ref{thm:A} but does not satisfy properties (1) and (3).

Theorem~\ref{thm:A} shows that even very strong additional assumptions on
a purely loxodromic action (including discreteness, acylindricity,
having quasiconvex orbits and having asymptotic translation lengths of
loxodromic elements being separated away from $0$) are, in general, not sufficient
to imply that the orbit map is a quasi-isometric embedding.

We briefly describe the construction of $Y$ in Theorem~\ref{thm:A} here. 
We start with an infinite sequence $v_n(a,b)\in F(a,b)$ (where $n=1,2,\dots $)  of distinct positive $7$-aperiodic words, that is such that no $v_n$ contains a subword of the form $u^7$ for any nontrivial $u$. We put $w_n=v_n(a,b)c\in F(a,b,c)$. Let $K$ be the set of all positive words $z\in F(a,b,c)$ such that $z$ is a subword of $w_n^m$ for some $m,n\ge 1$. Note that $\{a,b,c\}\subseteq K$. 
Then $Y$ is the Cayley graph of $F(a,b,c)$ with respect to the generating set  $K$. One can also view $Y$ as a "coned-off" version of the Cayley graph $X$ of $F(a,b,c)$ with respect to $\{a,b,c\}$ where for every $n\ge 1$ and for every conjugate $w_n'$ of $w_n$ in $F(a,b,c)$ we "cone-off" the axis $L(w_n')\subseteq X$ of $w_n'$ in $X$. See Definition~\ref{def:Y} below for details. 
The fact that we are coning off a collection of uniformly quasiconvex subsets of a hyperbolic graph $X$ implies (by \cite[Proposition 2.6]{KR}) that $Y$ is Gromov-hyperbolic and that part (4) of Theorem~\ref{thm:A} holds. Part (4) in turn easily implies part (3) since $F(a,b)\le F(a,b,c)$ is a quasiconvex (even convex for $X$) subgroup.  It is also clear from the construction that the orbit map $F(a,b)\to Y$, $g\mapsto gp$, is not a quasi-isometric embedding, and that in fact the action of $F(a,b)$ on $Y$ is not proper. 

To see that the action of $F(a,b)$ on $Y$ is purely loxodromic and
that has the asymptotic translation length of nontrivial elements of
$F(a,b)$ bounded below by $1/7$, we develop a precise  formula for
computing distances in $Y$ and exploit the $7$-aperiodicity property
of the words $v_n(a,b)$.  Note that the action of $F(a,b,c)$ on $X$ is
acylindrical,  but we are coning off a collection of subsets of $X$
that are uniformly quasiconvex but are not "geometrically separated"
in the sense of \cite{DGO}. The reason is that the axes of conjugates
of distinct $w_n$ and $w_m$ in $X$ can have arbitrarily long overlaps
as $n,m\to \infty$. Thus we cannot use the general result, given by
Proposition 5.40 of \cite{DGO}, to conclude that the action of
$F(a,b,c)$ on $Y$ is acylindrical (which may still be true). Instead
we give a direct argument, again exploiting the properties of periodic
and aperiodic words in free groups, that the action of $F(a,b)$ on $Y$
is acylindrical.  It would be interesting to understand whether, when
starting with an acylindrical $G$-action on a Gromov-hyperbolic space,
coning-off a $G$-equivariant collection of uniformly quasiconvex
subsets (perhaps with appropriate extra assumptions on various
constants) always produces an acylindrical action of $G$ on the
coned-off space.

I am grateful to Chris Leininger, Paul Schupp, Denis Osin, Michael Hull, Funda Gultepe, Kasra Rafi, Samuel Taylor and Spencer Dowdall for useful discussions.

\section{Construction and basic properties of the graph $Y$}

Let $F_3=F(a,b,c)$ and let $X$ be the Cayley graph of $F_3$ with
respect to the free basis $A=\{a,b,c\}$.

For a word $v$ in some alphabet, we denote by $|v|$ the length of $v$. For an element $g\in F(a,b,c)$ we denote by $|g|_A$ the freely reduced length of $g$ with respect to $A$ and denote by $||g||_A$ the cyclically reduced length of $g$ with respect to $A$. Note that $||g||_A=||g||_X$, the asymptotic translation length for the action of $g$ on $X$.

When dealing with words over the alphabet $A^{\pm 1}$, we will use $\equiv$ to indicate graphical equality of such words and we will use $=$ to indicate that the words represent the same element of $F(a,b,c)$.

We say that a freely reduced word $v\in F(a,b,c)$ is
\emph{$7$-aperiodic} if there does not exist a nontrivial cyclically
reduced word $u\in F(a,b,c)$ such that $u^7$ is a subword of $v$. It is well-known that there exist infinite $7$-aperiodic subsets of $F(a,b)$. For a sample reference we can use a result of Ol'shanskii, Lemma 1.2 in \cite{O}, where an infinite $7$-aperiodic set with additional small cancellation properties is constructed:

\begin{prop}\label{p:olsh}\cite[Lemma 1.2]{O}
There exists a sequence $v_n(a,b)\in F(a,b)$, where $n=1,2,3,\dots $
of positive words $v_n$ in $F(a,b)$ with the following
properties:

\begin{enumerate}
\item We have $|v_n|\to \infty$ as $n\to\infty$ and $|v_n|\ne |v_m|$
  whenever $m\ne n$.

\item Each $v_n$ is $7$-aperiodic.

\item If $u$ is a subword of some $v_n$ with $|u|\ge |v_n|/1000$ then
 $u$ occurs as a subword in $v_n$ exactly once, and $u$ does not
 occur as a subword of any $v_m$ with $m\ne n$.
\end{enumerate} 

\end{prop} 

Although we don't actually use part (3) of the above proposition in this paper, we record part (3) since it may be useful for further sharpening of the results obtained here.

\begin{conv}\label{conv:main}
From now and for the remainder of the paper, we fix a sequence of
positive words $v_n\in
F(a,b)$ satisfying the conclusions of Proposition~\ref{p:olsh}.

For $n=1,2,3,\dots $ put $w_n:=v_nc\in F(a,b,c)$.

Note that the words $v_n, w_n$ are positive and thus are freely and
cyclically reduced.

\end{conv}

\begin{df}[The graph $Y$]\label{def:Y}
Let $v_n\in F(a,b), w_n\in F(a,b,c)$,  where $n=1,2,3\dots$ be as 
in Convention~\ref{conv:main}. We define a graph $Y$ as follows.

The graph $X$ is a subgraph of $Y$ and $VY=VX$. The extra edges added
to $X$ to obtain $Y$ are defined as follows:

For every $n\ge 1$ and every conjugate $w_n'$ of $w_n$ in $F(a,b,c)$
we take the line $L(w_n')\subseteq X$ to be the axis of $w_n'$ when acting
on $X$; for every pair of vertices $x,y\in L(w_n')$ such that $d_X(x,y)\ge
2$ we add an edge joining $x$ and $y$. We call edges of $Y-X$
\emph{special edges}.

Since $X$ is the Cayley graph of $F(a,b,c)$, every oriented edge $e$
of $X$ already has a label $\mu(e)\in A^{\pm 1}$. If $e$ is an
oriented edge of $Y-X$ from a vertex $x\in VX$ to a vertex $y\in VX$,
then  $x,y\in F(a,b,c)$ and the geodesic segment $[x,y]_X$ is labelled
by the freely reduced form $z$ of the element $x^{-1}y\in
F(a,b,c)$. We then put $\mu(e):=z$ and $\mu(e^{-1})=z^{-1}$.

Thus $Y$ is a labelled graph where every oriented edge $e$ of $Y$ has
a label $\mu(e)$ which is a nontrivial freely reduced word in
$F(a,b,c)$. This assignment satisfies $\mu(e^{-1})=\mu(e)^{-1}$.
Moreover, every special oriented edge $e$ of $Y$ is labelled by some
nontrivial subword of some $w_n^m$. 

We equip $Y$ with the simplicial metric $d_Y$.  Note that the set of lines $L(w_n')$, as $n=1, 2, 3,\dots $ and $w_n'$ varies over all conjugates of $w_n$ in $F(a,b,c)$, is $F(a,b,c)$-invariant.  Hence the translation action of $F(a,b,c)$ on $X$ naturally extends to an action of $F(a,b,c)$ on $Y$ by graph automorphisms, and thus by $d_Y$-isometries. 

\end{df}

If $\gamma=e_1,e_2,\dots, e_k$ is an edge-path in $Y$, we put
$\mu(\gamma)\equiv \mu(e_1)\dots \mu(e_k)\in (A^{\pm 1})^\ast$. Note
that the label $\mu(\gamma)$ need not be a freely reduced word even if
the path $\gamma$ is a geodesic in $Y$.

Note that the space $X$ is Gromov-hyperbolic, and line each
$L(w_n')\subseteq X$ is a $0$-quasiconvex subset of $X$. 
Therefore the following statement is a direct corollary of
Proposition~2.6 of~\cite{KR} (see also Proposition 7.12 in \cite{Bow} for a related statement):

\begin{prop}\label{prop:KR}
There exist integer constants $\delta\ge 1$ and $C\ge 1$ such that:

\begin{enumerate}
\item The space $(Y,d_Y)$ is $\delta$-hyperbolic.

\item For any $x,y\in X$, if $\alpha=[x,y]_X$ is a $d_X$-geodesic from
  $x$ to $y$ in $X$ and $\beta=[x,y]_Y$ is a $d_Y$-geodesic from $x$
  to $y$ in $Y$ then $\alpha$ and $\beta$ are $C$-Hausdorff close with
  respect to $d_Y$.
\end{enumerate}

\end{prop}

\begin{conv}
For the remainder of the paper, we fix a number $C\ge 1$ satisfying the conclusion of Proposition~\ref{prop:KR}.
\end{conv}

We record the following useful immediate corollary of part (2) of
Proposition~\ref{prop:KR}:

\begin{cor}\label{cor:canc}
Let $x,y\in VX$ and let $x'$ be a vertex of $X$ such that $x'\in
[x,y]_X$. Then
\[
\left| d_Y(x,x')+d_Y(x',y)-d_Y(x,y)\right|\le 2C.
\] 
\end{cor}

\begin{prop}\label{prop:qc}
For any point $x\in Y$, the orbit  $F(a,b)x\subseteq Y$ is a quasiconvex
subset of $Y$
\end{prop}

\begin{proof}

We may assume that $x=1\in F(a,b)$, so that
$F(a,b)x =F(a,b)\subseteq VY$. 

Let $g\in F(a,b)$ be arbitrary and let $\alpha=[1,g]_X$ be the (unique)
$d_X$-geodesic path from $1$ to $g$ in $X$. Thus $\gamma$ is labbeled by the
freely reduced $v(a,b)$ form of $g$.  Let $\beta=[1,g]_Y$ be a $d_Y$-geodesic
from $1$ to $g$ in $Y$.

By Proposition~\ref{prop:KR}, for every point $p\in \beta$ there
exists a vertex $q$ on $\alpha$ such that $d_Y(p,q)\le C+1$. Thus $q$
represents an elememnt of $f(a,b,c)$ given by some initial segment of
the word $v(a,b)$ and hence $q\in F(a,b)$. This shows that $F(a,b)$ is
a $(C+1)$-quasiconvex subset of $Y$, as required.

\end{proof}

\section{Computing distances in $Y$}

\begin{df}
A nontrivial freely reduced word $z\in F(a,b,c)$ is said to be a
\emph{$\mathcal W$-word} if for some $n\ge 1$ and some integer
$m\ne 0$ the word $z$ is
a subword of $w_n^m$. 

For a freely reduced word $w\in F(a,b,c)$, a \emph{$\mathcal
  W$-decomposition} of $w$ is a decomposition
\[
w\equiv z_1\dots z_k
\]
such that each $z_i$ is a $\mathcal W$-word.
\end{df}

\begin{rk} Note that since each of the positive words $v_n(a,b)$ is $7$-aperiodic and $|v_n|\to\infty$ as $n\to\infty$, each of the letters $a,b$ appears in $v_n$ for all sufficiently large $n$. Also, by definition $w_n=v_nc$. Hence every letter from $\{a,b,c\}^{\pm 1}$ is a $\mathcal W$-word.  

Let $Z$ be the set of all positive $\mathcal W$-words $z\in F(a,b,c)$.  Then the graph $Y$ can also be viewed as the Cayley graph of $F(a,b,c)$ with respect to the generating set $Z$.
\end{rk}

\begin{lem}\label{lem:sub}
Let $z(a,b)\in F(a,b)$ be a nontrivial freely reduced word. Then $z$ is a $\mathcal W$-word if and only if there is $n\ge 1$ such that $z$ is a subword of $v_n$ or of $v_n^{-1}$.
\end{lem}
\begin{proof}

If $z(a,b)$ is a $\mathcal W$-word and thus a subword of some $w_n^m=(v_n(a,b)c)^m$ (where $m\in \mathbb Z \setminus \{0\}$) then, since $z$ does not involve $c^{\pm 1}$ it follows that  $z$ is a subword of $v_n$ or of $v_n^{-1}$. The statement of the lemma now follows.

\end{proof}

\begin{notation}
For $g\in F(a,b,c)$ denote $|g|_Y:=d_Y(1,g)$.
\end{notation}

\begin{lem}[Distance formula]\label{lem:dist}
Let $w\in F(a,b,c)$ be a nontrivial freely reduced word.

Then $|w|_Y$ is equal to the smallest $k\ge 1$ such that there
exists a $\mathcal
  W$-decomposition $w\equiv z_1\dots z_k$.
\end{lem}

\begin{proof}
The definition of $Y$ implies that if $z\in F(a,b,c)$ is a $\mathcal
W$-word, then for every $g\in F(a,b,c)$ we have $d_Y(g,gz)=1$. Thus if $w\equiv z_1\dots z_t$ is a $\mathcal W$-decomposition, then
$|w|_Y\le t$. 

Suppose now that $\gamma=e_1 e_2 \dots e_k$ is a $d_Y$-geodesic
edge-path from $1$ to $w$ in $Y$, where $k=|w|_Y$. Put $u_i=\mu(e_i) \in F(a,b,c)$. Then
$w=_{F(a,b,c)} u_1u_2\dots u_k$, and each $u_i$ is a $\mathcal W$-word.

After freely reducing the product $u_1u_2\dots u_k$ we get a
factorization $w\equiv z_1\dots z_r$ where $r\le k$ and each $z_i$ is
the remainder of exactly one of the $u_j$ after all the free
cancelations are performed. Thus each $z_i$ is a $\mathcal W$-word as
well, and $w\equiv z_1\dots z_r$ is a $\mathcal W$-decomposition.
Hence, by the argument above, $k=|w|_Y\le r$. Thus $k=r$ and we
have found a $\mathcal W$-decomposition $w\equiv z_1\dots z_k$ with
$k=|w|_Y$.

We have already seen that if $w$ has a $\mathcal W$-decomposition with
$t$ factors, then $|w|_Y\le t$.

Therefore $|w|_Y$ is equal to the smallest number of factors among
all $\mathcal W$-decompositions of $w$, as required.

\end{proof}

\begin{prop}\label{prop:power}
Let $1\ne g\in F(a,b)$ be arbitrary. Then:

\begin{enumerate}
\item For every $n\ge 1$ we have $|g^n|_Y\ge \lfloor \frac{n}{7}\rfloor$. 
\item We have $||g||_Y\ge \frac{1}{7}$. 
\end{enumerate}

\end{prop}

\begin{proof}
Let $g\equiv u w u^{-1}$ where $u,w\in F(a,b)$ are freely reduced and $w$ is cyclically reduced. Then the freely reduced form of $g^n$ is $uw^nu^{-1}$.

Let $uw^nu^{-1}\equiv z_1\dots z_k$ be a $\mathcal W$-factorization of the word $uw^nu^{-1}$. Thus each $z_i$ is a $\mathcal W$-word and $z_i\in F(a,b)$. Hence by Lemma~\ref{lem:sub}, each $z_i$ is a subord of some $v_{n_i}^{\pm 1}$.  Since the words $v_j(a,b)$ are 7-aperiodic, it follows that for every subword of $uw^nu^{-1}$ of the form $w^7$ this subword nontrivially overlaps at least two distinct factors $z_i$. Therefore $k\ge  \lfloor \frac{n}{7}\rfloor$.

Hence, by the distance formula provided by Lemma~\ref{lem:dist}, for every $n\ge 1$ we have $|g^n|_Y\ge \lfloor \frac{n}{7}\rfloor$. The definition of $||g||_Y$ now implies that $||g||_Y\ge \frac{1}{7}$. 

\end{proof}

\section{Acylindricity}

The following useful fact is a special case of Lemma~4 of Lyndon-Sch\"utzenberger~\cite{LSz}:
\begin{lem}\label{lem:LSz}
Let $u_1,u_2\in F(a,b,c)$ be nontrivial cyclically reduced words such that for some $k,t\ge 1$ the words $u_1^k$ and $u_2^t$ have a common initial segment of length $\ge |u_1|+|u_2|$. Then there exists a unique root-free cyclically reduced word $u_0\in F(a,b,c)$ such that $u_1\equiv u_0^r$ and $u_2\equiv u_0^s$ for  some $r,s\ge 1$. 
\end{lem}

\begin{lem}\label{lem:use}
Let $R\ge 1$ and let $L\ge 100(R+4C)(R+6C+10)$.

Let $h\in F(a,b,c)$ be a freely reduced word and let $g\equiv \alpha^{-1}u\alpha\in F(a,b,c)$ be a freely reduced word with $u$ being cyclically reduced. 

Suppose that $|h|_Y\ge L$, $|g|_Y\le R$ and $|h g h^{-1}|_Y\le R$. Then $h\equiv h_0\sigma_1\sigma_2 u^k \alpha$ where:
\begin{enumerate}
\item We have $|k|\ge 100(R+6C+1)$.
\item $\sigma_1, \sigma_2$ are subwords of $\alpha^{-1} u^{\pm 1}\alpha$.
\item We have  $|h_0|_Y, |\sigma_1|_Y, |\sigma_2|_Y \le R+4C$. 
\end{enumerate}
\end{lem}
\begin{proof}

Let $k\in \mathbb Z$ be the largest in the absolute value integer such that the freely reduced word $h\in F(a,b,c)$ ends in $u^k\alpha$, where $k=0$ corresponds to the case where $h$ does not end in $u^{\pm 1}\alpha$.   It is not hard to see, by a variation of the argument below, that $k=0$ is not possible under the assumptions of this lemma, so we can write $h$ as $h\equiv h_1 u^k \alpha$. We will assume that $k>0$ as the case $k<0$ is similar.

Then, at the level of group elements, in $F(a,b,c)$ we have \[hgh^{-1} =h_1\alpha (\alpha^{-1} u^k \alpha) (\alpha^{-1} u\alpha)  (\alpha^{-1} u^{-k} \alpha) \alpha^{-1} h_1^{-1}=h_1\alpha  (\alpha^{-1} u\alpha)   \alpha^{-1} h_1^{-1}.\]  Put $h_2=h_1\alpha \in F(a,b,c)$, so that $h_2$ is a freely reduced word. The maximal choice of $k$ implies that in freely reducing the product $h_2 \cdot (\alpha^{-1} u\alpha) \cdot h_2^{-1}$ not all of the word $\alpha^{-1} u\alpha$ cancels. Hence the freely reduced form of $hgh^{-1}$ is graphically equal to $h_3 u_1 h_4^{-1}$ where $u_1$ is a subword of $\alpha^{-1}u\alpha$, where $h_2\equiv h_3 \tau$ with $\tau^{-1}$ being an initial segment of $\alpha^{-1}u\alpha$ and where $h_2\equiv h_4\nu$ with $\nu^{-1}$ being a terminal segment of $\alpha^{-1} u\alpha$. 
We can express $h_1\equiv h_5\rho$, where $\rho^{-1}$ is a maximal initial segment of $\alpha$ that cancels in the product $h_1\alpha$, with $\alpha\equiv \rho^{-1}\alpha_1$. 
Then $h_2\equiv h_5\alpha_1\equiv h_3 \tau$ and $h_2\equiv h_5\alpha_1 \equiv h_4\nu$.  Recall also that the freely reduced form of $hgh^{-1}$ is graphically equal to $h_3 u_1 h_4^{-1}$. Hence there exist subwords $\sigma_1, \dots, \sigma_4$ and $\beta_1,\dots, \beta_4$ of  $\alpha^{-1} u^{\pm 1}\alpha$ such that $h_1\equiv h_6\sigma_1\sigma_2\equiv h_7\sigma_3\sigma_4$ such that the freely  reduced form of  $hgh^{-1}$ is graphically equal to $h_6\beta_1\beta_2 u_1 \beta_3^{-1} \beta_4^{-1} h_7^{-1}$.
Recall also that $u_1$ is  a subword of $\alpha^{-1}u\alpha$ and that $h\equiv h_1 u^k \alpha$.

By assumption, $|h g h^{-1}|_Y\le R$. Since the freely reduced form of $h g h^{-1}$ is $h_6\beta_1\beta_2 u_1 \beta_3^{-1} \beta_4^{-1} h_7^{-1}$, Corollary~\ref{cor:canc}  implies that $|h_6|_Y, |h_7|_Y \le R+4C$.    Since $\sigma_1, \sigma_2, \alpha$ are subwords of the freely reduced word $g=\alpha^{-1}u\alpha$, and since by assumption $|g|_Y\le R$, Corollary~\ref{cor:canc}  implies that $|\sigma_1|_Y, |\sigma_2|_Y, |\alpha|_Y\le R+4C$.  We also have $h\equiv h_1 u^k \alpha\equiv h_6\sigma_1\sigma_2 u^k \alpha$, and by assumption $|h|_Y\ge L$.
By the triangle inequality we now get $|u^k|_Y\ge L-4(R+4C)$.  Since $|g|_Y\le R$, Corollary~\ref{cor:canc}  implies that $|u|_Y\le R+4C$. Thus 
\[
L-4(R+4C)\le |u^k|_Y \le k (R+4C)
\]
and hence $k\ge (L-4(R+4C))/(R+4C)=\frac{L}{R+4C}-4\ge 100(R+6C+1)$, where the last inequality holds by the assumption on $L$.
Thus the factorization  $h\equiv h_6\sigma_1\sigma_2 u^k \alpha$ satisfies all the requirements of the lemma.

\end{proof}

\begin{prop}\label{prop:ac}
Let $R\ge 1$ and $L\ge 100(R+4C)(R+6C+10)$.   Let $g,g'\in F(a,b,c)$ be nontrivial freely reduced words conjugate in $F(a,b,c)$ to some elements of $F(a,b)$,  and let $h\in F(a,b,c)$ be such that $|h|_Y\ge L$, $|g|_Y, |g'|_Y\le R$ and that $d_Y(h,gh), d_Y(h, g'h)\le R$. Then there exists a root-free nontrivial freely reduced $g_0\in F(a,b,c)$ such that $g=g_0^t$, $g'=g_0^r$, where $1\le |r|, |t|\le 7(R+4C+1)$. 

\end{prop}
\begin{proof}
We have $d_Y(h,gh)=|h^{-1}gh|_Y, d_Y(h, g'h)=|h^{-1}g'h|\le R$.  Write $g$ as a freely reduced word $g\equiv \alpha^{-1}u\alpha\in F(a,b)$,  with $u\in F(a,b)$ being cyclically reduced. Similarly, write $g'$ as a freely reduced word $g'\equiv (\alpha')^{-1}u'\alpha'\in F(a,b)$,  with $u'\in F(a,b)$ being cyclically reduced. 

Applying Lemma~\ref{lem:use} we conclude that there exist factorizations $h^{-1}\equiv h_0\sigma_1\sigma_2 u^k \alpha$ and $h^{-1}\equiv h_0'\sigma_1'\sigma_2' (u')^r \alpha'$ where $|k|, |r|\ge 100(R+6C+1)$, where $\sigma_1, \sigma_2$ are subwords of $g$, where where $\sigma_1', \sigma_2'$ are subwords of $g'$, and where $|h_0|_Y, |h_0'|_Y, |\sigma_1|_Y, |\sigma_2|_Y, |\sigma_1'|_Y, |\sigma_2'|\le R+4C$.

We now see how the subwords $u^k$ and $(u')^s$ overlap in \[h^{-1}\equiv h_0\sigma_1\sigma_2 u^k \alpha \equiv h_0'\sigma_1'\sigma_2' (u')^s \alpha'.\]

{\bf Case 1. } Suppose first that the length of the overlap between  $u^k$ and $(u')^s$ is $<|u|+|u'|$. Without loss of generality we may assume that $|u'|\le |u|$ and that $k,r>0$.

Then either $u^{k-2}$ is a subword of  $h_0'\sigma_1'\sigma_2' $,  or $u^{k-2}$ is a subword of $\alpha'$,  or $(u')^r$ is contained in $u^k$. 

Recall that $k,r\ge 100(R+6C+1)$. 

If $u^{k-2}$ is a subword of  $h_0'\sigma_1'\sigma_2' $ then Corollary~\ref{cor:canc}  implies that $|u^{k-2}|_Y\le  |h_0'\sigma_1'\sigma_2'|_Y+4C\le 3(R+4C)+4C=3R+16C$.   Since $u\in F(a,b)$, Proposition~\ref{prop:power} implies that $|u^{k-2}|_Y\ge (k-2)/7-1$. Hence $(k-2)/7-1\le |u^{k-2}|_Y\le  3R+16C$ and $k\le 7(3R+16C+1)+2$, yielding a contradiction. 

If $u^{k-2}$ is a subword of $\alpha'$, then Corollary~\ref{cor:canc}  implies that $|u^{k-2}|_Y\le |\alpha'|_Y+4C\le R+6C$. Since $|u^{k-2}|_Y\ge (k-2)/7-1$, we get $(k-2)/7-1\le |u^{k-2}|_Y\le R+6C$ and $k\le 7(R+6C+1)+2$, again yielding a contradiction with $k\ge 100(R+6C+1)$.

Suppose now that $(u')^r$ is contained in $u^k$. Since $|u'|\le |u|$ and the length of the overlap between $u^k$ and $(u')^r$ is $<|u|+|u'|$, it follows that $(u')^r$ is contained in some subword $u^2$ or $u^k$. Hence either $u^{k/4}$ is a subword of $h_0'\sigma_1'\sigma_2' $  or $u^{k/4}$ is a subword of $\alpha'$. We then again obtain a contradiction by a similar argument to above.

{\bf Case 2.} Suppose now that the length of the overlap between  $u^k$ and $(u')^s$ is $\ge |u|+|u'|$.  Without loss of generality we may assume that $|\alpha|\le |\alpha'|$.

Assume first that $|\alpha|= |\alpha'|$, so that $\alpha'=\alpha$. Then Lemma~\ref{lem:LSz} implies that there exists a cyclically reduced word $u_0\in F(a,b)$ such that $u=u_0^t$ and $u'=u_0^s$, so that $g=(\alpha^{-1} u_0\alpha)^t$ and $g'=(\alpha^{-1} u_0\alpha)^s$.
By assumption $|g|_Y, |g'|_Y\le R$ which by Corollary~\ref{cor:canc} implies that $|u_0^t|_Y, |u_0^s|_Y\le R+4C$. Hence by Proposition~\ref{prop:power} we have $|t|/7-1, |s|/7-1\le R+4C$ and hence $|t|,|s|\le 7(R+4C+1)$, as required. The conclusion of the proposition is established in this case.

Assume now that $|\alpha|<|\alpha'|$. Let $u_\ast$ be the cyclic permutation of $u$ such that the overlap between $u^k$ and $(u')^r$ in $h_1^{-1}$ ends in $u_\ast$.
Lemma~\ref{lem:LSz} implies that there exists a cyclically reduced word $u_0\in F(a,b)$ such that $u_\ast=u_0^t$ and $u'=u_0^s$.  We may assume (after possibly replacing $u_0$ by its inverse) that $rs>0$.
The fact that $|\alpha|<|\alpha'|$ now implies that the first letter of $\alpha'$ is the same as the first letter of $u_0$. This contradicts the fact that the word $g'\equiv (\alpha')^{-1} (u')^r \alpha =(\alpha')^{-1} u_0^{rs} \alpha'$ is freely reduced as written.  Thus Case~2 cannot happen, which completes the proof of the proposition.

\end{proof}
 
 \begin{cor}\label{cor:ac}
 The action of $F(a,b)$ on $Y$ is acylindrical. 
 \end{cor}
\begin{proof}

It is enough to check the acylindricity condition for the vertices of $Y$.

Let $R\ge 1$. Put $L=L(R):=100(R+4C)(R+6C+10)$ and $M=M(R):=14(R+4C+1)+1$.

Let $x,y\in VY=F(a,b,c)$ be vertices such that $d_Y(x,y)\ge L$. 
Put \[S=\{g\in F(a,b)| d_Y(x,gx)\le R, d_Y(y,gy)\le R\}.\]

We claim that $\#(S)\le M$. 

We have $d_Y(x,y)=d_Y(1,x^{-1}y)\ge L$.  Let $g\in F(a,b)$ be such
that $d_Y(x,gx)\le R, d_Y(y,gy)\le R$.
Then for $g_1=x^{-1}gx$ we have $|g_1|_Y=|x^{-1}gx|_Y=d_Y(x,gx)\le R$ and 
\[
d_Y(x^{-1}y, g_1 x^{-1}y)=|y^{-1}x^{-1} g_1 x^{-1}y|_Y=|y^{-1}x^{-1} x^{-1}gx x^{-1}y|_Y=|y^{-1}gy|_Y=d_Y(y,gy)\le R.
\]

Put $h=x^{-1}y\in F(a,b,c)$, so that $|h|_Y=d_Y(x,y)\ge L$. 
Also put 
\begin{gather*}
S_1:=\{g_1\in F(a,b,c)\big| \\ |g_1|_Y\le R, |h^{-1}g_1h|_Y\le R, \text{ and $g_1$ is conjugate to an element of $F(a,b)$ in $F(a,b,c)$}\}.
\end{gather*}

Since $x^{-1}Sx\subseteq S_1$, to verify the claim above it is enough to show that $\#(S_1)\le M$.
 
Suppose $\#(S_1)\ge 2$. Let $1\ne g_1\in S_1$. We can uniquely express $g_1$ as $g_1=g_0^t$ where $g_0\in F(a,b,c)$ is a nontrivial root-free element and $t\ge 1$. Now if $g_2\in S_1$ is an arbitrary nontrivial element, then Proposition~\ref{prop:ac} implies that $g_2=g_0^s$ where $|s|\le 7(R+4C+1)$.  It follows that $\#(S_1)\le M$, as required.

\end{proof}

We now summarize the properties of the action of $F(a,b)$ on $Y$:

\begin{thm}\label{thm:main}
The following hold:

\begin{enumerate}
\item The graph $Y$ is Gromov-hyperbolic and $F(a,b)$ acts on $Y$ by simplicial isometries.
\item The action of $F(a,b)$ on $Y$ is acylindrical.
\item The action of $F(a,b)$ on $Y$ is purely loxodromic, that is, every $1\ne g\in F(a,b)$ acts on $Y$ as a loxodromic isometry.
\item For every $1\ne g\in F(a,b)$ we have $||g||_Y\ge 1/7$.
\item For any $p\in Y$ the orbit $F(a,b)p\subseteq Y$ is a quasiconvex subset of $Y$.
\item There exists $C\ge 1$ such that for any $x,y\in F(a,b)$ if $\alpha_{x,y}$ is a geodesic from $x$ to $y$ in the Cayley graph of $F(a,b)$ with respect to the basis $\{a,b\}$, and if $\beta=[x,y]_Y$ is a geodesic from $x$ to $y$ in $Y$, then $\alpha$ and $\beta$ are $C$-Hausdorff close in $Y$.
\item For any $p\in Y$, the orbit map $F(a,b)\to Y$, $g\mapsto gp$, is not a quasi-isometric embedding. Moreover, the action of $F(a,b)$ on $Y$ is not proper.
\end{enumerate}
\end{thm}

\begin{proof}

Parts (1) and (6) are established in Proposition~\ref{prop:KR}.  Part (2) is Corollary~\ref{cor:ac} above. Part (4) is Proposition~\ref{prop:power}, and part (4) directly implies part (3).  Part (5) is Proposition~\ref{prop:qc}.

To see that (7) holds, note that for every $n\ge 1$ $v_n(a,b)$ is a $\mathcal W$-word and hence, by definition of $Y$, we have $d_Y(1,v_n)=|v_n(a,b)|_Y=1$.
On the other hand, $v_n$ is a freely reduced word in $F(a,b)$ with $|v_n|\to\infty$ as $n\to\infty$. This shows, with $p=1\in VY$, that the orbit map $F(a,b)\to Y$,$g\mapsto gp$ is not a quasi-isometric embedding and that the action of $F(a,b)$ on $Y$ is not proper.
\end{proof}

\end{document}